\documentclass[12pt,reqno]{amsart}
\usepackage{amsmath,amssymb,amsfonts,amscd,latexsym,amsthm,mathrsfs,verbatim,mathdots}
\usepackage{tikz-cd}
\usepackage{hyperref}
\theoremstyle{plain}
\newtheorem{theorem}{Theorem}[section]
\newtheorem{lemma}[theorem]{Lemma}
\newtheorem{proposition}[theorem]{Proposition}

\theoremstyle{definition}
\newtheorem{definition}[theorem]{Definition}

\theoremstyle{remark}
\newtheorem{remark}[theorem]{Remark}

\newcommand{\C}{\mathbf{C}}
\newcommand{\Z}{\mathbf{Z}}
\newcommand{\Q}{\mathbf{Q}}
\newcommand{\Qb}{\overline{\Q}}

\newcommand{\PP}{\mathbf{P}}
\newcommand{\h}{\mathcal{H}}

\newcommand{\SL}{\mathrm{SL}}
\newcommand{\abcd}[4]{\begin{pmatrix}#1&#2\\#3& #4\end{pmatrix}}

\newcommand{\BB}{\mathcal{B}}
\newcommand{\CC}{\mathcal{C}}
\newcommand{\Div}{\mathrm{Div}}
\newcommand{\ord}{\mathrm{ord}}
\newcommand{\dv}{\mathrm{div}}
\newcommand{\DD}{\mathcal{D}}
\newcommand{\Gal}{\mathrm{Gal}}
\newcommand{\CH}{\mathrm{CH}}

\newcommand{\rvline}{\hspace*{-\arraycolsep}\vline\hspace*{-\arraycolsep}}

\begin{document}

\title{Bernoulli determinants and cuspidal subgroups}

\author{François Brunault}
\address{Unité de mathématiques pures et appliquées (UMPA), ENS de Lyon, 46 allée d'Italie, 69007 Lyon, France}
\email{francois.brunault@ens-lyon.fr}


\date{\today}


\begin{abstract}
We give an explicit formula for the order of the rational cuspidal class group of the modular curve $X_1(N)$ for an arbitrary integer~$N$. The proof relies on results of Streng on the group of modular units on $X_1(N)$, and requires computing a certain determinant involving the second Bernoulli polynomial.
We also define a higher weight analogue of the cuspidal class group and speculate that its order is related to a similar determinant defined using a higher degree Bernoulli polynomial.
\end{abstract}

\maketitle

\section{Introduction} \label{sec introduction}

\subsection{Statement of the main results} \label{subsec main results}

For a positive integer $N$, let $X_1(N)$ be the modular curve over $\Q$ associated with the congruence subgroup
\begin{equation*}
\Gamma_1(N) = \left\{\abcd{a}{b}{c}{d} \in \SL_2(\Z) : a \equiv d \equiv 1 \bmod{N}, \, c \equiv 0 \bmod{N}\right\}.
\end{equation*}
Let $J_1(N)$ be the Jacobian variety of $X_1(N)$. The \emph{cuspidal subgroup} $\CC_1(N)$ is the subgroup of $J_1(N)(\C)$ consisting of the linear equivalence classes of divisors of degree~0 on the cusps of $X_1(N)(\C)$. By the Manin-Drinfeld theorem \cite{Dri73}, the group $\CC_1(N)$ is finite. We define the \emph{rational cuspidal class group} $\CC_1^\Q(N)$ as the subgroup of $\CC_1(N)$ consisting of the classes of Galois invariant divisors of degree $0$ on the cusps of $X_1(N)(\C)$ (see Section~\ref{sec cuspidal subgroup}).\footnote{The analogue of $\CC_1^\Q(N)$ for the modular curve $X_0(N)$ is called the rational cuspidal divisor class group in \cite{Yo23a,Yo23b}.} By definition, the group $\CC_1^\Q(N)$ is contained in $\CC_1(N)(\Q) := \CC_1(N) \cap J_1(N)(\Q)$.

Our main result is a formula for the order of $\CC_1^\Q(N)$. Before stating our formula, let us introduce some notation. For an integer $k \geqslant 1$, we denote by $B_k(x)$ the $k$th Bernoulli polynomial, so that $B_1(x) = x - \frac12$, $B_2(x) = x^2 - x + \frac16$, etc. Given a Dirichlet character $\chi$ modulo $N$, the generalised Bernoulli number $B_{k,\chi}$ is defined as
\begin{equation} \label{def Bkchi}
B_{k,\chi} = \sum_{t = 1}^N B_k\Bigl(\frac{t}{N}\Bigr) \chi(t).
\end{equation}
The conductor of $\chi$ will be denoted by $N_\chi$. Let $P_N^+$ (resp. $P_N^-$) be the set of even (resp. odd) primitive Dirichlet characters of conductor dividing $N$. For a positive integer $N$ and a prime number $p$, we write $N = N_p N^{(p)}$, where $N_p = p^{v_p(N)}$ is the $p$-primary part of $N$. Finally, we denote by $\varphi$ the Euler totient function, and by $\sigma_0(n)$ the number of positive divisors of $n$.

\begin{theorem} \label{main thm C1N}
For any integer $N \geqslant 5$, we have
\begin{equation*}
\begin{split}
|\CC_1^\Q(N)| & = \frac{\gcd(2,N)^2 \cdot 9}{2^{N-5} N} \prod_{\substack{d | N \\ d \geqslant 3}} d^{\varphi(d)/2} \cdot \prod_{p | N} \frac{p^2}{p^2-1} \\
& \qquad \times \prod_{\chi \in P_N^+} \biggl( B_{2,\chi}^{\sigma_0(N/N_\chi)} \Bigl(\frac{N}{N_\chi}\Bigr)^{\frac12 \sigma_0(N/N_\chi)} \prod_{\substack{p \mid N \\ p \nmid N_\chi}} (1-\overline{\chi}(p) p^{-2})^{v_p(N) \sigma_0(N^{(p)}/N_\chi)} \biggr),
\end{split}
\end{equation*}
where in the products $p$ runs over the prime factors of $N$.
\end{theorem}

We refer to Section \ref{subsec known results} below for a comparison between this theorem and previously known results.
 
The generalised Bernoulli number $B_{2,\chi}$ has a simple expression in terms of the Dirichlet $L$-value $L(\chi,-1)$ (see Lemma \ref{lem Bk Lchi}), so that Theorem \ref{main thm C1N} can also be stated in terms of $L$-values.

The most important ingredient in the proof of Theorem \ref{main thm C1N} is a recent result of Streng \cite{Str23} giving generators for the group of modular units on $X_1(N)$ (see Section~\ref{sec cuspidal subgroup}). These generators are expressed in terms of the classical Siegel functions, whose divisors involve the polynomial $B_2(x)$. The computation of $|\CC_1^\Q(N)|$ then reduces to that of an explicit determinant $D_{2,N}$ involving $B_2(x)$, and this determinant turns out to generalise naturally to any Bernoulli polynomial:
\begin{equation*}
D_{k,N} = \begin{cases}
\displaystyle  \det \left( B_k\left(\left\{\frac{ij}{N}\right\}\right) \right)_{0 \leqslant i,j \leqslant \lfloor N/2 \rfloor} \qquad \quad \;\; (k \geqslant 2 \textrm{ even}), \\
\displaystyle \det \left( B_k\left(\left\{\frac{ij}{N}\right\}\right) \right)_{1 \leqslant i,j \leqslant \lfloor (N-1)/2 \rfloor} \qquad (k \geqslant 3 \textrm{ odd}),
\end{cases}
\end{equation*}
where $\lfloor x \rfloor$ denotes the integer part of $x$, and $\{x\} = x - \lfloor x \rfloor$ is the fractional part.

This whole paper actually originates from the consideration of the determinant $D_{3,N}$, which appears in \cite{BdJLRV24} in connection with the $K_2$ group of families of elliptic curves having a rational point of order $N$. The main result \cite[Theorem~4.13]{BdJLRV24} giving the linear independence of $K_2$ elements crucially relies on the non-vanishing of $D_{3,N}$ for $N=7,8,10$ (for more on this, see the discussion at the end of Section \ref{sec higher}).

In Theorem \ref{thm formula DkN}, we give an explicit formula for $D_{k,N}$ up to sign, showing in particular that it is non-zero. The formula involves the numbers $B_{k,\chi}$ where the characters $\chi$ have conductor dividing $N$ and satisfy $\chi(-1)=(-1)^k$. The idea behind the computation is to view $D_{k,N}$ as a kind of group determinant, and decompose the matrix into eigenspaces indexed by Dirichlet characters.

In the recent work \cite{Liu26}, Liu computed several families of determinants, including the family $D_{k,N}$. The formula obtained involves the $L$-values $L(\chi,k)$ and has a simpler form \cite[Theorem 3.2]{Liu26}. The proof is different and uses a technique of gcd filtering with respect to the rows and columns of the matrix defining $D_{k,N}$.

In Section \ref{sec higher}, we propose a definition of a higher weight cuspidal class group, which is inspired by Beilinson's construction of the Eisenstein symbol. This subgroup lives in a certain higher Chow group of a Kuga-Sato variety. We give some evidence as to why we expect its order to be related to $D_{k,N}$.

\subsection{Comparison with known results} \label{subsec known results}

The cuspidal subgroup $\CC_1(N)$ has a natural subgroup defined as follows. Let $\CC_1^0(N)$ be the image in $\CC_1(N)$ of the group of degree zero divisors supported on the cusps of $X_1(N)$ lying over the cusp $0$ of $X_0(N)$. Since these cusps are rational, the group $\CC_1^0(N)$ is contained in $\CC_1^\Q(N)$. Let $p$ be a prime. The order of $\CC_1^0(p^n)$ was determined by Klimek \cite{Kli75} for $n=1$ and by Kubert-Lang \cite[Chapter 6, Theorem 3.4]{KL81} for any $n$. For $p \geqslant 5$ and $n=1$, the formula reads
\begin{equation} \label{eq C1p}
|\CC_1^0(p)| = \frac{p^{(p-1)/2}}{2^{p-3}} \prod_{\substack{\chi \in P_p^+ \\ \chi \neq 1}} B_{2,\chi}.
\end{equation}
Lupoian \cite[Section 6.3]{Lup25} proved that $\CC_1^0(p) = \CC_1^{\Q}(p) = \CC_1(p)(\Q)$ \footnote{The references \cite{Lup25, Yan09, YY10} use a model of $X_1(N)$ for which the cusp $\infty$ is defined over $\Q$, while in our model the cusp $0$ is defined over $\Q$. The Atkin-Lehner involution $\tau \mapsto -1/N\tau$ switches the two models and swaps $0$ and $\infty$, so the stated equality of groups is equivalent to the one in \cite{Lup25}.}
and the formula \eqref{eq C1p} is equivalent to Theorem \ref{main thm C1N}. The formula \eqref{eq C1p} was generalised to $\CC_1^0(N)$ for any $N$ by Yu \cite{Yu80}. Later, Takagi \cite{Tak92, Tak93, Tak95} computed the order of the full cuspidal subgroup $\CC_1(p^n)$ for $p \neq 2$ and $n \geqslant 1$. Takagi \cite{Tak14} also gave a formula for the order of $\CC_1^\Q(2p)$, which is equivalent to the one given in Theorem \ref{main thm C1N}.

The cuspidal subgroup is of special interest due to its connection with the arithmetic of $J_1(N)$. When $N = p$ is prime, Conrad, Edixhoven and Stein conjectured that the rational torsion subgroup $J_1(p)(\Q)_{\mathrm{tors}}$ is equal to $\CC_1^0(p)$ (and thus to $\CC_1^\Q(p)$) \cite[Conjecture 6.2.2]{CES03}. This was proved by Ohta up to $2$-primary torsion \cite{Oht13}. The latter result is an analogue of a famous result of Mazur \cite{Maz77}, who proved that the cuspidal subgroup of the modular curve $X_0(p)$ is equal to the rational torsion subgroup of the Jacobian $J_0(p)$. For arbitrary $N$, Yoo conjectured that the group $J_0(N)(\Q)_{\mathrm{tors}}$ is equal to the rational cuspidal divisor class group of $X_0(N)$ \cite[Conjectures 1.2 and 1.3]{Yo23b} and proved the equality of the $p$-primary parts for any prime $p \geqslant 5$ with $p^2 \nmid N$ \cite[Theorem 1.4]{Yo23b}. In these works, a key role is played by the Eisenstein ideal in the Hecke algebra acting on the space of weight $2$ modular forms. 

Note that Theorem \ref{main thm C1N} says nothing about the group structure of $\CC_1^\Q(N)$. Yang and Yu \cite{Yan09, YY10} determined the structure of the $p$-primary part of $\CC_1^0(p^n)$ when $p$ is a regular prime. Lupoian \cite[Section 7]{Lup25} investigated the full structure of $\CC_1^\Q(p)$ and $\CC_1(p)$. Yoo \cite{Yo23a} completely determined the structure of the rational cuspidal divisor class group of $X_0(N)$ for any $N$. It would be interesting to see whether the product decomposition appearing in Theorem \ref{main thm C1N} reflects, at least partly, the structure of the cuspidal class group.

The determinant $D_{k,N}$ is an example of a paratrophic determinant, a general notion introduced by Frobenius generalising group determinants (see \cite{Con98}, \cite[Section~2]{Ste22}). If $k$ is even, then $D_{k,N}$ is a Frobenius determinant for the multiplicative semigroup $(\Z/N\Z)/\pm 1$. In this case, it may be possible to deduce Theorem \ref{thm formula DkN} from the formula for semigroup determinants established in \cite[Theorem 6.14]{Ste22}. If $k$ is even and $N=p$ is prime, then $D_{k,p}$ can actually be computed directly using a group determinant for $(\Z/p\Z)^\times/\pm 1$.

For $k=4$, the determinant $D_{4,N}$ appears in \cite[Lemma 3.4.2]{MS89} in relation to the construction of elements in the motivic cohomology group $H^3_{\mathcal{M}}(\operatorname{Sym}^2 E, \Q(3))_{\Z}$ for the symmetric square of elliptic curves $E$ over $\Q$.

The rational cuspidal class group of $X_1(N)$ also appears in a recent work by Huber, Opoku and Ye \cite{HOY26}, in connection with the construction of modular forms for $\Gamma_1(N)$ by means of infinite $q$-products.

\subsection{Organisation of the paper}

In Section \ref{sec bernoulli}, we introduce the Bernoulli transform, which is a linear operator associated with the Bernoulli polynomial $B_k(x)$. It decomposes into eigenspaces indexed by the Dirichlet characters modulo $N$. We compute a matrix in each eigenspace in Section \ref{sec matrix}, which leads to the computation of $D_{k,N}$ in Section \ref{sec proof}. We determine the order of the rational cuspidal class group in Section \ref{sec cuspidal subgroup}. Finally, in Section \ref{sec higher}, we define and discuss a weight $k$ analogue of the cuspidal class group.

\,

\textbf{Acknowledgements.}
I thank the anonymous referee of \cite{BdJLRV24} for raising the question of the non-vanishing of $D_{3,N}$ for general $N$, which led to the present paper. I also thank Elvira Lupoian for interesting discussions about the cuspidal subgroup and its generalisations, Hang Liu for sharing ideas on Bernoulli determinants, Benjamin Steinberg for explaining his formula for semigroup determinants, and Iván Rosas-Soto for answering my questions on étale motivic cohomology.

I was funded by the research project “Motivic homotopy, quadratic invariants and diagonal classes” (ANR-21-CE40-0015) operated by the French National Research Agency (ANR).

\section{The Bernoulli transform} \label{sec bernoulli}

Let $k \geqslant 2$ and $N \geqslant 1$ be integers. For $t \in \Z/N\Z$, we define $B_{k,N}(t) = B_k(\tilde{t}/N)$, where $\tilde{t}$ is the unique representative of $t$ such that $0 \leqslant \tilde{t} \leqslant N-1$. Let $V_N$ denote the $\Q$-vector space of functions $f \colon \Z/N\Z \to \Q$. Consider the linear transformation $\BB_{k,N} \colon V_N \to V_N$ defined by
\begin{equation*}
(\BB_{k,N} f) (x) = \sum_{y \in \Z/N\Z} B_{k,N}(xy) f(y),
\end{equation*}
which we call the \emph{Bernoulli transform}.

Let $V_N = V_N^+ \oplus V_N^-$ be the decomposition of $V_N$ into even and odd functions. Since $B_k(1-x) = (-1)^k B_k(x)$ and $B_k(0)=B_k(1)$, the function $B_{k,N}$ is even if $k$ is even, and odd if $k$ is odd. It follows that if $k$ is even, then $\BB_{k,N}$ induces a map $\BB_{k,N}^+ \colon V_N^+ \to V_N^+$. Similarly, if $k$ is odd, then $\BB_{k,N}$ induces a map $\BB_{k,N}^- \colon V_N^- \to V_N^-$.

\begin{lemma} \label{lem det Bk}
We have
\begin{equation*}
\det \BB_{k,N}^{(-1)^k} = 2^{\lfloor (N-1)/2 \rfloor} D_{k,N}.
\end{equation*}
\end{lemma}

\begin{proof}
If $k$ is even, then a basis of $V_N^+$ is given by the functions $[i] + [-i]$ with $0 \leqslant i \leqslant \lfloor N/2 \rfloor$, where $[i]$ is the characteristic function of the class of $i$ in $\Z/N\Z$. The matrix of $\BB_{k,N}^+$ in this basis is the matrix defining $D_{k,N}$, except that the rows corresponding to the indices $1 \leqslant i \leqslant \lfloor (N-1)/2 \rfloor$ have been multiplied by $2$.

If $k$ is odd, then a basis of $V_N^-$ is given by the functions $[i] - [-i]$ with $1 \leqslant i \leqslant \lfloor (N-1)/2 \rfloor$, and the matrix of $\BB_{k,N}^-$ in this basis is twice the matrix defining $D_{k,N}$.
\end{proof}

There is an action of $(\Z/N\Z)^\times$ on $V_N$ given by $(\alpha \cdot f)(x) = f(\alpha x)$. Let $V_{N,\C} = V_N \otimes_\Q \C$ be the complex vector space of functions $f \colon \Z/N\Z \to \C$. There is a decomposition into eigenspaces
\begin{equation*}
V_{N,\C} = \bigoplus_{\chi} V_{N,\chi}
\end{equation*}
where $\chi$ runs over the Dirichlet characters modulo $N$, and $V_{N,\chi}$ is the subspace of functions $f$ satisfying $\alpha \cdot f = \chi(\alpha) f$ for every $\alpha \in (\Z/N\Z)^\times$. Note that $V^+_{N,\C}$ (resp. $V^-_{N,\C}$) is the direct sum of the subspaces $V_{N,\chi}$ with $\chi$ even (resp. $\chi$ odd).

Because of the relation $\BB_{k,N} (\alpha \cdot f) = \alpha^{-1} \cdot \BB_{k,N} f$, the operator $\BB_{k,N}$ maps $V_{N,\chi}$ into $V_{N,\overline{\chi}}$. We will need an explicit basis of $V_{N,\chi}$. Let $N_\chi$ be the conductor of $\chi$, and $\chi_0$ be the primitive Dirichlet character modulo $N_\chi$ induced by $\chi$. For every integer $M$ satisfying $N_\chi \mid M \mid N$, define the function $f_{\chi,M} \in V_{N,\C}$ by
\begin{equation*}
f_{\chi,M}(x) = \begin{cases}
0 & \textrm{if } \frac{N}{M} \nmid x, \\
\chi_M\bigl(\frac{x}{N/M}\bigr) & \textrm{if } \frac{N}{M} \mid x,
\end{cases}
\end{equation*}
where $\chi_M$ is the Dirichlet character modulo $M$ induced by $\chi_0$.

\begin{lemma} \label{lem basis Vchi}
The functions $(f_{\chi,M})_{N_\chi \mid M \mid N}$ form a $\C$-basis of $V_{N,\chi}$.
\end{lemma}

\begin{proof}
The set $\Z/N\Z$ is the disjoint union of the subsets $S_M = \frac{N}{M} \cdot (\Z/M\Z)^\times$ with $M \mid N$, and $S_M$ is the orbit of $\frac{N}{M}$ under the action of $(\Z/N\Z)^\times$ on $\Z/N\Z$ by multiplication. Therefore, any function $f \in V_{N,\chi}$ can be written uniquely as a linear combination $f = \sum_{M \mid N} f_M$, where $f_M$ is supported on $S_M$. Moreover, $f_M$ belongs to the $\chi$-eigenspace and is thus determined by its value at $\frac{N}{M}$. Since the conductor of $\chi$ is $N_\chi$, we have $f_M = 0$ for $N_\chi \nmid M$, and $f_M$ is a multiple of $f_{\chi,M}$ if $N_\chi \mid M$.
\end{proof}

Thanks to Lemma \ref{lem det Bk}, Theorem \ref{thm formula DkN} reduces to computing the determinants of the induced maps
\begin{equation*}
\BB_{k,\chi} : V_{N,\chi} \to V_{N,\overline{\chi}}
\end{equation*}
in the bases of Lemma \ref{lem basis Vchi}.

\section{Computing the matrix of the Bernoulli transform} \label{sec matrix}

In this section, we compute the matrix of $\BB_{k,\chi}$ in the bases provided by Lemma~\ref{lem basis Vchi}. We fix a Dirichlet character $\chi$ modulo $N$ having the same parity as $k$. For each integer $M$ with $N_\chi | M | N$, let us write
\begin{equation*}
\BB_{k,\chi} f_{\chi, M} = \sum_{N_\chi \mid M' \mid N} c_{M,M'} f_{\overline{\chi}, M'}.
\end{equation*}
Note that
\begin{equation*}
c_{M,M'} = (\BB_{k,\chi} f_{\chi,M}) \Bigl( \frac{N}{M'} \Bigr).
\end{equation*}
Therefore
\begin{align*}
c_{M,M'} & = \sum_{y \in \Z/N\Z} B_{k,N}\Bigl(\frac{N}{M'} y\Bigr) f_{\chi,M}(y) \\
& = \sum_{z = 0}^{M-1} B_{k,M}\Bigl( \frac{N}{M'} z \Bigr) \chi_M(z),
\end{align*}
where we used the substitution $y = (N/M) z$ with $0 \leqslant z \leqslant M-1$. Let $\delta = \gcd(MM',N)$. Note that $M' \mid \delta$, so that we have a well-defined reduction map $\pi \colon \Z/M\Z \to \Z/(\frac{MM'}{\delta}) \Z$. Then
\begin{equation} \label{eq cMM' 2}
c_{M,M'} = \sum_{t \in \Z/(\frac{MM'}{\delta})\Z} B_{k, MM'/\delta}\Bigl( \frac{N}{\delta} t \Bigr) S(t),
\end{equation}
where
\begin{equation*}
S(t) = \sum_{\substack{z \in \Z/M\Z \\ \pi(z) = t}} \chi_M(z).
\end{equation*}

\begin{lemma} \label{lem St}
If $t \not\in (\Z/(\frac{MM'}{\delta}) \Z)^\times$, then $S(t) = 0$.
If $t \in (\Z/(\frac{MM'}{\delta}) \Z)^\times$, then
\begin{equation*}
S(t) = \begin{cases}
0 & \textrm{if } N_\chi \nmid \frac{MM'}{\delta}, \\
\frac{\varphi(M)}{\varphi(MM'/\delta)} \chi_{MM'/\delta}(t) & \textrm{if } N_\chi \mid \frac{MM'}{\delta}.
\end{cases}
\end{equation*}
\end{lemma}

\begin{proof}
If $t \not\in (\Z/(\frac{MM'}{\delta}) \Z)^\times$, then no lift of $t$ in $\Z/M\Z$ is invertible, so $S(t) = 0$.

Assume $t \in (\Z/(\frac{MM'}{\delta}) \Z)^\times$. Write $\pi^\times \colon (\Z/M\Z)^\times \to (\Z/(\frac{MM'}{\delta}) \Z)^\times$ for the surjective group morphism induced by $\pi$. Let $\tilde{t}$ be any lift of $t$ in $(\Z/M\Z)^\times$. Then
\begin{equation*}
S(t) = \sum_{\substack{z \in (\Z/M\Z)^\times \\ \pi^\times(z) = t}} \chi_M(z) = \chi_M(\tilde{t}) \sum_{h \in \ker(\pi^\times)} \chi_M(h).
\end{equation*}
The kernel of $\pi^\times$ is contained in the kernel of $\chi_M$ if and only if $\chi_M$ factors through $(\Z/(\frac{MM'}{\delta}) \Z)^\times$, which is equivalent to the condition $N_\chi \mid \frac{MM'}{\delta}$. In this case,
\begin{equation*}
S(t) = \chi_M(\tilde{t}) \cdot |\ker(\pi^\times)| = \chi_{MM'/\delta}(t) \frac{\varphi(M)}{\varphi(MM'/\delta)}.
\end{equation*}
Otherwise, there must exist $h \in \ker(\pi^\times)$ such that $\chi_M(h) \neq 1$, and $S(t) = 0$.
\end{proof}

By Lemma \ref{lem St}, we see that $c_{M,M'} = 0$ if $N_\chi \nmid \frac{MM'}{\delta}$. From now on, we assume that $N_\chi \mid \frac{MM'}{\delta}$. By \eqref{eq cMM' 2}, we have
\begin{equation*}
c_{M,M'} = \frac{\varphi(M)}{\varphi(MM'/\delta)} \sum_{t \in (\Z/(\frac{MM'}{\delta})\Z)^\times} B_{k, MM'/\delta}\Bigl( \frac{N}{\delta} t \Bigr) \chi_{MM'/\delta}(t).
\end{equation*}
Since $N/\delta$ is invertible in $\Z/(\frac{MM'}{\delta})\Z$, we get
\begin{equation} \label{eq cMM' 3}
c_{M,M'} = \frac{\varphi(M)}{\varphi(MM'/\delta)} \overline{\chi}_{MM'/\delta}(N/\delta) \sum_{t \in (\Z/(\frac{MM'}{\delta})\Z)^\times} B_{k, MM'/\delta}(t) \chi_{MM'/\delta}(t).
\end{equation}

Since $N/\delta$ is prime to $N_\chi$, we have $\overline{\chi}_{MM'/\delta}(N/\delta) = \overline{\chi}_0(N/\delta)$. Moreover, we recognise in \eqref{eq cMM' 3} a generalised Bernoulli number (see equation \eqref{def Bkchi}). The following classical lemma shows that it can be expressed in terms of a Dirichlet $L$-value. Recall that the $L$-function associated with a Dirichlet character $\psi : (\Z/m\Z)^\times \to \C^\times$ is defined by $L(\psi,s) = \sum_{n=1}^\infty \psi(n)/n^s$ for $s \in \C$, $\mathrm{Re}(s)>1$, with the convention $\psi(n)=0$ if $\gcd(n,m) \neq 1$. It admits a meromorphic continuation to the complex plane \cite[p.~20]{Car92}.

\begin{lemma} \label{lem Bk Lchi}
Let $\chi$ be a Dirichlet character modulo $N$, and $\chi_0$ be the associated primitive character, of conductor $N_\chi$. Let $Q \geqslant 1$ be an integer such that $N_\chi \mid Q \mid N$, and $\chi_Q$ be the Dirichlet character modulo $Q$ induced by $\chi$. Then
\begin{equation} \label{eq sum Bk}
B_{k, \chi_Q} = - k Q^{1-k} L(\chi_Q, 1-k).
\end{equation}
Moreover, we have
\begin{equation} \label{eq BkchiQ}
B_{k,\chi_Q} = \Bigl(\frac{Q}{N_\chi}\Bigr)^{1-k} B_{k,\chi_0} \prod_{p \mid Q} (1- \chi_0(p) p^{k-1}) \neq 0.
\end{equation}
\end{lemma}

\begin{proof}
The identity \eqref{eq sum Bk} is \cite[(1.75)]{Car92}.
Since Dirichlet $L$-functions admit an Euler product \cite[(1.78)]{Car92}, the $L$-function $L(\chi_Q, s)$ is obtained from $L(\chi_0, s)$ by removing the Euler factors at the primes $p$ dividing $Q$, which gives
\begin{equation*}
L(\chi_Q, 1-k) = L(\chi_0, 1-k) \cdot \prod_{p \mid Q} (1- \chi_0(p) p^{k-1}).
\end{equation*}
Applying \eqref{eq sum Bk} with $Q=N_\chi$, we have
\begin{equation*}
B_{k, \chi_0} = -k N_\chi^{1-k} L(\chi_0, 1-k),
\end{equation*}
which gives the first part of \eqref{eq BkchiQ}.
Finally, to prove the non-vanishing of $B_{k,\chi_Q}$, it suffices to prove $L(\chi_0, 1-k) \neq 0$. This follows from the functional equation of $L(\chi_0, s)$ \cite[Theorem 9.9]{AIK14} and the fact that $L(\overline{\chi}_0,k) \neq 0$ thanks to the convergent Euler product.
\end{proof}

Using \eqref{eq cMM' 3} and Lemma \ref{lem Bk Lchi}, we arrive at the following conclusion.

\begin{proposition} \label{prop cMM'}
Let $M \geqslant 1$ be an integer such that $N_\chi \mid M \mid N$. Then
\begin{equation*}
\BB_{k,\chi} f_{\chi, M} = N_\chi^{k-1} B_{k,\chi_0} \sum_{N_\chi \mid M' \mid N} \tilde{c}_{\chi,M,M'} f_{\overline{\chi}, M'}
\end{equation*}
with
\begin{equation} \label{eq cMM' 4}
\tilde{c}_{\chi,M,M'} = \begin{cases}
0 & \textrm{if } N_\chi \nmid \frac{MM'}{\delta}, \\
\frac{\varphi(M)}{\varphi(MM'/\delta)} \overline{\chi}_0(N/\delta) \bigl(\frac{MM'}{\delta}\bigr)^{1-k} \displaystyle \prod_{p \mid MM'/\delta} (1- \chi_0(p) p^{k-1}) & \textrm{if } N_\chi \mid \frac{MM'}{\delta},
\end{cases}
\end{equation}
where $\delta = \gcd(MM',N)$.
\end{proposition}

\section{Computation of $D_{k,N}$} \label{sec proof}

In this section, we give an explicit formula for the Bernoulli determinant $D_{k,N}$ (Theorem \ref{thm formula DkN}), showing in particular that it is non-zero. As explained in Section~\ref{sec bernoulli}, it suffices to consider the maps $\BB_{k,\chi} \colon V_{N,\chi} \to V_{N,\overline{\chi}}$ for each Dirichlet character $\chi$ modulo $N$ having the same parity as $k$. Thanks to Proposition \ref{prop cMM'}, this reduces to computing the determinant of the matrix
\begin{equation*}
A_\chi = (\tilde{c}_{\chi,M,M'})_{M,M' \in \DD_\chi},
\end{equation*}
where $\DD_\chi$ is the set of integers $M \geqslant 1$ such that $N_\chi \mid M \mid N$, and $\tilde{c}_{\chi,M,M'}$ is defined by \eqref{eq cMM' 4}.

The crucial observation is that $\tilde{c}_{\chi,M,M'}$ decomposes as a product over the prime factors of $N$. For an integer $Q \geqslant 1$ and a prime number $p$, we write $Q = Q_p Q^{(p)}$, where $Q_p = p^{v_p(Q)}$ is the $p$-part of $Q$. Then the set $\DD_\chi$ is in bijection with $\prod_{p \mid N} \DD_{\chi,p}$, where $\DD_{\chi,p}$ is the set of integers $M_p$ such that $N_{\chi,p} \mid M_p \mid N_p$. Moreover, $A_\chi$ decomposes as a Kronecker product over the prime factors of $N$:
\begin{equation} \label{eq Achip}
A_\chi = \bigotimes_{p \mid N} A_{\chi,p}.
\end{equation}
Given square matrices $A_1, \ldots, A_r$ of sizes $n_1,\ldots,n_r$, the determinant of their Kronecker product is given by
\begin{equation*}
\det(A_1 \otimes \cdots \otimes A_r) = \prod_{i=1}^r (\det A_i)^{\prod_{j \neq i} n_j}.
\end{equation*}
Applying this to \eqref{eq Achip}, we get
\begin{equation} \label{eq det Achi}
\det(A_\chi) = \prod_{p \mid N} \det(A_{\chi,p})^{\sigma_0(N^{(p)}/N_\chi^{(p)})},
\end{equation}
where $\sigma_0(Q)$ denotes the number of positive divisors of $Q$.

We now compute the determinant of $A_{\chi,p}$ for each prime factor $p$ of $N$. Write $N_p = p^n$ and $N_{\chi,p} = p^\ell$, so that $A_{\chi,p}$ is of the form $(a_{m,m'})_{\ell \leqslant m, m' \leqslant n}$.

\,

\textbf{Case 1: $p$ divides $N_\chi$.}

For given $m,m' \in \{\ell, \ldots, n\}$, let $M_p = p^m$, $M'_p = p^{m'}$ and $\delta_p = \gcd(M_p M'_p, N_p)$. We have
\begin{equation*}
v_p\Bigl(\frac{M_p M'_p}{\delta_p}\Bigr) = m+m' - \min(m+m', n) = \max(0, m+m'-n).
\end{equation*}
So for $m+m' < \ell+n$, we have
\begin{equation*}
v_p\Bigl(\frac{M_p M'_p}{\delta_p}\Bigr) < \ell = v_p(N_{\chi,p}),
\end{equation*}
so that $a_{m,m'} = 0$. It follows that the coefficients of $A_{\chi,p}$ above the anti-diagonal are zero. Using the formula \eqref{eq cMM' 4}, we obtain
\begin{align*}
\det(A_{\chi,p}) & = \pm \prod_{m+m' = \ell+n} a_{m,m'} = \pm \prod_{m=\ell}^n \frac{\varphi(p^m)}{\varphi(p^\ell)} p^{\ell (1-k)} = \pm \prod_{m=\ell}^n p^{m-\ell k} \\
& = \pm p^{(n-\ell+1)(\frac12 (\ell+n) - \ell k)}.
\end{align*}
In particular $A_{\chi,p}$ is invertible. Moreover, since $n-\ell+1 = \sigma_0(N_p/N_{\chi,p})$ and $\sigma_0$ is multiplicative, we obtain
\begin{equation} \label{eq det Achip 1}
\det(A_{\chi,p})^{\sigma_0(N^{(p)}/N_\chi^{(p)})} = \pm p^{\sigma_0(N/N_\chi) (\frac12 (\ell+n) - \ell k)} = \pm \bigl( N_p N_{\chi,p}^{1-2k} \bigr)^{\frac12 \sigma_0(N/N_\chi)}.
\end{equation}

\,

\textbf{Case 2: $p$ does not divide $N_\chi$.}

In this case $A_{\chi,p} = (a_{m,m'})_{0 \leqslant m,m' \leqslant n}$ with
\begin{equation*}
a_{m,m'} = \begin{cases}
\varphi(p^m) \overline{\chi}_0(p)^{n-m-m'} & \textrm{if } m+m' \leqslant n \\
p^{n-m' + (m+m'-n)(1-k)} (1-\chi_0(p) p^{k-1}) & \textrm{if } m+m' > n.
\end{cases}
\end{equation*}
We compute the determinant of $A_{\chi,p} = A_{\chi,p}^{(n)}$ by induction on $n$.

For $n=0$, we have $a_{0,0} = 1$, so that $\det(A_{\chi,p}^{(0)}) = 1$. Assume $n>0$. We have
\begin{equation*}
A_{\chi,p}^{(n)} = \begin{pmatrix}
\overline{\chi}_0(p)^n & \varphi(p) \overline{\chi}_0(p)^{n-1} & \cdots & \varphi(p^{n-1}) \overline{\chi}_0(p) & \varphi(p^n) \\
\overline{\chi}_0(p)^{n-1} & \varphi(p) \overline{\chi}_0(p)^{n-2} & \cdots & \varphi(p^{n-1}) & \\
\vdots & \vdots & \iddots & & \\
\vdots & \varphi(p) & & a_{m,m'} & \\
1 & & & &
\end{pmatrix}.
\end{equation*}
Let $L_0, \ldots, L_n$ be the rows of this matrix. The elementary operation $L_0 - \overline{\chi}_0(p) L_1 \to L_0$ puts the matrix $A_{\chi,p}^{(n)}$ in the form
\begin{equation*}
\begin{pmatrix}
0 & \cdots & 0 & \rvline & \alpha \\
\hline
& & & \rvline & * \\
& A_{\chi,p}^{(n-1)} & & \rvline & \vdots \\
& & & \rvline & *
\end{pmatrix},
\end{equation*}
with
\begin{equation*}
\alpha = \varphi(p^n) - \overline{\chi}_0(p) p^{n-1 + 1-k } (1 - \chi_0(p) p^{k-1}) = p^n (1-\overline{\chi}_0(p) p^{-k}).
\end{equation*}
An induction gives
\begin{equation*}
\det(A_{\chi,p}^{(n)}) = \pm p^{\frac12 n(n+1)} (1-\overline{\chi}_0(p) p^{-k})^n.
\end{equation*}
In particular $A_{\chi,p}$ is invertible. Moreover, since $n+1 = \sigma_0(N_p) = \sigma_0(N_p/N_{\chi,p})$, we obtain
\begin{align}
\nonumber \det(A_{\chi,p})^{\sigma_0(N^{(p)}/N_\chi^{(p)})} & = \pm p^{\sigma_0(N/N_\chi) \frac12 n} (1-\overline{\chi}_0(p) p^{-k})^{n \sigma_0(N^{(p)}/N_\chi^{(p)})} \\
\label{eq det Achip 2} & = \pm N_p^{\frac12 \sigma_0(N/N_\chi)} (1-\overline{\chi}_0(p) p^{-k})^{v_p(N) \sigma_0(N^{(p)}/N_\chi)}.
\end{align}

To evaluate $D_{k,N}$, it remains to put together the previous computations. Considering Lemma \ref{lem det Bk}, we need to compute the determinant of the linear map $\BB_{k,N}^{(-1)^k}$. We use the basis of $V_{N,\C}^{(-1)^k}$ obtained by concatenating the bases $(f_{\chi,M})_{M \in \DD_\chi}$ of Lemma \ref{lem basis Vchi}. After a suitable permutation of the basis vectors of the codomain of $\BB_{k,N}^{(-1)^k}$, the matrix of this linear map is block diagonal, and Proposition \ref{prop cMM'} gives
\begin{align}
\nonumber \det \BB_{k,N}^{(-1)^k} & = \pm \prod_{\chi(-1)=(-1)^k} \det \bigl( N_\chi^{k-1} B_{k,\chi_0} \tilde{c}_{\chi,M,M'} \bigr)_{M,M' \in \DD_\chi} \\
\label{eq det BkN} & = \pm \prod_{\chi(-1)=(-1)^k} (N_\chi^{k-1} B_{k,\chi_0})^{\sigma_0(N/N_\chi)} \det(A_\chi).
\end{align}
Putting together Lemma \ref{lem det Bk} and equations \eqref{eq det BkN}, \eqref{eq det Achi}, \eqref{eq det Achip 1}, \eqref{eq det Achip 2}, we obtain the following final formula (see Section \ref{subsec main results} for the notation).

\begin{theorem} \label{thm formula DkN}
Let $k \geqslant 2$ be an integer. Then
\begin{equation*}
\begin{split}
D_{k,N} = \pm 2^{-\lfloor \frac{N-1}{2} \rfloor} \prod_{\chi \in P_N^{(-1)^k}} & \biggl( B_{k,\chi}^{\sigma_0(N/N_\chi)} (N/N_\chi)^{\frac12 \sigma_0(N/N_\chi)} \\
& \qquad \times \prod_{\substack{p \mid N \\ p \nmid N_\chi}} (1-\overline{\chi}(p) p^{-k})^{v_p(N) \sigma_0(N^{(p)}/N_\chi)}\biggr).
\end{split}
\end{equation*}
\end{theorem}

\section{The rational cuspidal class group of $X_1(N)$} \label{sec cuspidal subgroup}

We first recall classical facts about modular curves and Siegel units (see \cite[Sections 1 and 2.1]{Kat04} for more details). For an integer $N \geqslant 1$, let $Y_1(N) = Y(1,N)$ be the modular curve over $\Q$ associated with the congruence subgroup $\Gamma_1(N)$, as defined in \cite[2.1]{Kat04}. If $N \geqslant 4$, the curve $Y_1(N)$ represents the functor mapping a $\Q$-scheme $S$ to the set of isomorphism classes of pairs $(E,P)$, where $E$ is an elliptic curve over $S$, and $P \in E(S)$ is a point of order $N$. There is a canonical isomorphism $\nu : \Gamma_1(N) \backslash \h \to Y_1(N)(\C)$ induced by the map $\tau \mapsto (\C/(\Z\tau+\Z), 1/N)$. Let $X_1(N)$ be the smooth compactification of $Y_1(N)$. We denote by $C = X_1(N) \setminus Y_1(N)$ the set of cusps, seen as closed points of $X_1(N)$. Let $\Div^0(C)$ be the group of divisors of degree $0$ on $C$, where the degree of a divisor $D = \sum_{x \in C} n_x [x]$ is defined as $\sum_{x \in C} n_x [\Q(x):\Q]$. By definition, the rational cuspidal class group $\CC_1^{\Q}(N)$ is the image of the Abel-Jacobi map $\mathrm{AJ} : \Div^0(C) \to J_1(N)(\Q)$. Note that $\CC_1^\Q(N)$ is contained in $\CC_1(N)(\Q)$. It is not known in general whether these two groups are equal (this is true when $N$ is a prime or twice a prime, see \cite[Section 6.3]{Lup25} and \cite[Theorem 1.1 and Section 1.3]{Tak14}).

\begin{lemma} \label{lem C1N}
The rational cuspidal class group $\CC_1^\Q(N)$ is isomorphic to the cokernel of the divisor map
\begin{equation} \label{eq div OY1N}
\dv \colon \mathcal{O}(Y_1(N))^\times \to \Div^0(C).
\end{equation}
\end{lemma}

\begin{proof}
There is an exact sequence
\begin{equation*}
0 \to \Qb^\times \to \mathcal{O}(Y_1(N)_{\Qb})^\times \xrightarrow{\dv} \mathcal{P} \to 0,
\end{equation*}
where $\mathcal{P}$ is the group of principal divisors supported on $C(\Qb)$. Taking invariants under $\Gal(\Qb/\Q)$ and using the fact that $H^1(\Gal(\Qb/\Q), \Qb^\times) = 0$ \cite[Lemma 4.3.7]{GS17}, we obtain the exact sequence
\begin{equation*}
0 \to \Q^\times \to \mathcal{O}(Y_1(N))^\times \xrightarrow{\dv} \mathcal{P}^{\Gal(\Qb/\Q)} \to 0.
\end{equation*}
After identifying $\Div^0(C)$ with $\Div^0(C(\Qb))^{\Gal(\Qb/\Q)}$, we deduce the exact sequence
\begin{equation} \label{eq seq cuspidal}
\mathcal{O}(Y_1(N))^\times \xrightarrow{\dv} \Div^0(C) \xrightarrow{\mathrm{AJ}} J_1(N)(\Q),
\end{equation}
which proves the lemma.
\end{proof}

The kernel of the divisor map \eqref{eq div OY1N} is $\Q^\times$. Letting $U = \mathcal{O}(Y_1(N))^\times / \Q^\times$, we obtain an injective map $U \to \Div^0(C)$, which has finite cokernel by Lemma \ref{lem C1N}. Therefore, it induces an isomorphism, which we still denote by $\dv$:
\begin{equation*}
\dv : U_\Q \xrightarrow{\cong} \Div^0(C)_\Q,
\end{equation*}
where, for any abelian group $A$, we write $A_\Q = A \otimes_\Z \Q$. Since $U$ and $\Div^0(C)$ are torsion-free, we will identify them with their respective images in $U_\Q$ and $\Div^0(C)_\Q$.

We now recall the classical Siegel units. For an integer $1 \leqslant a \leqslant N-1$, let
\begin{equation*}
g_{0,a}(\tau) = q^{1/12} \prod_{n=0}^{\infty} (1-q^n \zeta_N^a) \prod_{n=1}^{\infty} (1-q^n \zeta_N^{-a}) \qquad (\tau \in \h),
\end{equation*}
where $\h$ is the upper half-plane, $q^\alpha = e^{2\pi i \alpha\tau}$ and $\zeta_N = e^{2\pi i/N}$.
It is known that the function $g_{0,a}^{12N}$ is modular for the group $\Gamma(N)$ \cite[Chapter 2, Theorem 1.2]{KL81}.
In fact $g_{0,a}$ defines an element of $\mathcal{O}(Y(N))^\times_\Q$ \cite[1.4 and 1.9]{Kat04}, where $Y(N)$ is the modular curve over $\Q$ of full level $N$. The curve $Y_1(N)$ is isomorphic to a quotient $G \backslash Y(N)$ \cite[2.1]{Kat04} and $g_{0,a}$ is $G$-invariant by \cite[1.7(1)]{Kat04}. It follows that $g_{0,a}$ defines an element of $\mathcal{O}(Y_1(N))^\times_\Q$, and thus an element of $U_\Q$.

The units $g_{0,a}$ with $1 \leqslant a \leqslant \lfloor \frac{N}{2} \rfloor$ actually form a $\Q$-basis of $U_\Q$. A short proof can be found in \cite[Proposition 4.1]{Str23}.\footnote{More precisely, this proposition states that the functions $H_1, \ldots, H_{\lfloor N/2 \rfloor}$ defined in \cite[(1.4)]{Str23} are multiplicatively independent modulo $\C^\times$. An application of \cite[Lemma 2.13(6)]{Str23} with $(a_1,a_2)=(0,a/N)$ and $M = (\begin{smallmatrix} 0 & -1 \\ 1 & 0 \end{smallmatrix})$ shows that the function $H_a(\tau)$ coincides with $g_{0,a}(-1/\tau)$ up to a multiplicative constant, which implies the claim.} Proposition \ref{pro index 1} below, together with the non-vanishing of the determinant $D_{2,N}$ (Theorem \ref{thm formula DkN}), will give an independent proof of this fact.

Recall that two subgroups $H, H'$ of a $\Q$-vector space $V$ are said to be commensurable if there exists a subgroup $H_0$ of $V$ which has finite index in both $H$ and $H'$. If this is the case, we define the (generalised) index $(H : H') = (H : H_0) (H' : H_0)^{-1}$, where $H_0$ is any such subgroup.

Let $S$ be the subgroup of $U_\Q$ generated by the Siegel units $g_{0,a}$ with $1 \leqslant a \leqslant \lfloor \frac{N}{2} \rfloor$. By Lemma \ref{lem C1N}, we have
\begin{align}
\nonumber |\CC_1^\Q(N)| & = (\Div^0(C) : \dv(U)) \\
\label{eq DSU} & = (\Div^0(C) : \dv(S)) (\dv(S) : \dv(U)),
\end{align}
where the indices are well-defined because the $g_{0,a}$ form a $\Q$-basis. We will deal with each index separately.

Let us first recall the computation of the divisor of a Siegel unit. The set of cusps of $X_1(N)(\C)$ can be identified with $\Gamma_1(N) \backslash \PP^1(\Q)$. Under this identification, the action of $\mathrm{Aut}(\C)$ on the cusps is given explicitly by \cite[Lemma 5.9]{Bru20}. We deduce the following description of $C$ as well as the associated residue fields.

\begin{lemma} \label{lem cusps}
The map sending an integer $0 \leqslant k \leqslant \lfloor \frac{N}{2} \rfloor$ to the Galois orbit of $1/k$ induces a bijection between $\{0, \ldots, \lfloor \frac{N}{2} \rfloor\}$ and $C$. Moreover, the residue field $F_k$ of the cusp $1/k$ is given by:
\begin{itemize}
\item If $k \neq 0, \frac{N}{2}$, then $F_k = \Q(\zeta_{\gcd(k,N)})$.
\item If $k = 0$, then $F_0 = \Q(\zeta_N)^+$ is the maximal real subfield of $\Q(\zeta_N)$.
\item If $N$ is even and $k = \frac{N}{2}$, then $F_{N/2} = \Q(\zeta_{N/2})^+$ is the maximal real subfield of $\Q(\zeta_{N/2})$.
\end{itemize}

\end{lemma}

\begin{proof}
Let $k$ be an integer with $0 \leqslant k \leqslant \lfloor \frac{N}{2} \rfloor$. Using the notation of \cite[Lemma 5.9]{Bru20}, the cusp $1/k \in \PP^1(\Q)$ corresponds to the pair $(\bar{k},1)$, where $\bar{k}$ is the image of $k$ in $\Z/N\Z$. Moreover, the residue field $F_k$ corresponds to the stabiliser of $(\bar{k},1)$ in $\mathrm{Aut}(\C)$. This stabiliser can be computed using \cite[Lemma 5.9]{Bru20} and this gives the second assertion. Finally, let $x = p/q \in \PP^1(\Q)$ be an arbitrary cusp, with $\gcd(p,q)=1$. Then $x$ is in the Galois orbit of $1/k$, where $k$ is the unique integer in $\{0, \ldots, \lfloor \frac{N}{2} \rfloor\}$ such that $\bar{q} = \pm \bar{k}$. This proves the first assertion.
\end{proof}

\begin{lemma} \label{lem div ga}
Let $1 \leqslant a \leqslant \lfloor \frac{N}{2} \rfloor$ and $0 \leqslant k \leqslant \lfloor \frac{N}{2} \rfloor$. Then the order of vanishing of $g_{0,a}$ at the cusp $\frac{1}{k}$ is given by
\begin{equation*}
\ord_{1/k} (g_{0,a}) = \frac{N}{2 \gcd(k,N)} B_{2,N}(ak),
\end{equation*}
unless $N=4$ and $k=2$, in which case $\ord_{1/2}(g_{0,1}) = -1/24$ and $\ord_{1/2}(g_{0,2}) = 1/12$.
\end{lemma}

\begin{proof}
For $N \geqslant 5$, this is \cite[Lemma 5.11]{Bru20}. The proof uses the fact that the width of the cusp $1/k \in X_1(N)(\C)$ is $N/\gcd(k,N)$, which is true as long as $(N,k) \neq (4,2)$ \cite[Proposition 6.3.20.(a)]{CS17}, so the proof extends to this case. By \emph{loc.~cit.}, the width of the cusp $1/2$ of $X_1(4)(\C)$ is 1, which gives the remaining case.
\end{proof}

\begin{proposition} \label{pro index 1}
For $N \neq 1,2,4$, we have
\begin{equation*}
(\Div^0(C) : \dv(S)) = |D_{2,N}| \cdot \frac{\gcd(2,N)}{2^{\lfloor N/2 \rfloor+1}} \prod_{\substack{d \mid N \\ d \geqslant 3}} d^{\varphi(d)/2} \cdot \frac{24}{N^2} \prod_{p \mid N} \frac{p^2}{p^2-1}.
\end{equation*}
\end{proposition}

\begin{proof}
In what follows, we write $r = \lfloor \frac{N}{2} \rfloor$. We identify $\Div(C)_\Q$ with $\Q^{r+1}$ by means of the basis $([1/k])_{0 \leqslant k \leqslant r}$. Let $M$ be the $(r+1) \times r$ matrix with column vectors $\dv(g_{0,1}), \ldots, \dv(g_{0,r})$. The matrix $M$ can be computed using Lemma \ref{lem div ga} and is related to the Bernoulli determinant $D_{2,N}$. More precisely, let $\eta = \Bigl(\begin{smallmatrix} \eta_0 \\ \vdots \\ \eta_r \end{smallmatrix}\Bigr) \in \Q^{r+1}$ be the vector with $\eta_k = \frac{N}{12 \gcd(k,N)}$. Note that $\deg(\eta) > 0$. Then by Lemma \ref{lem div ga}, we have for $N \neq 4$:
\begin{equation*}
\det (\eta | M) = \det \Bigl( \frac{N}{2 \gcd(k,N)} B_{2,N}(ak) \Bigr)_{0 \leqslant k,a \leqslant r} = 2^{-r-1} \Bigl(\prod_{k=0}^r \frac{N}{\gcd(k,N)}\Bigr) D_{2,N}.
\end{equation*}
Distinguishing the cases where $N$ is even or odd, the product over $k$ equals
\begin{equation*}
\prod_{k=0}^r \frac{N}{\gcd(k,N)} = \gcd(2,N) \prod_{\substack{d \mid N \\ d \geqslant 3}} d^{\varphi(d)/2}.
\end{equation*}

On the other hand,
\begin{align*}
(\Div^0(C) : \dv(S)) & = (\Z \eta \oplus \Div^0(C) : \Z \eta \oplus \dv(S)) \\
& = (\Div(C) : \Z \eta \oplus \dv(S)) \cdot (\Div(C) : \Z \eta \oplus \Div^0(C))^{-1} \\
& = |\det(\eta | M)| \cdot \deg(\eta)^{-1}.
\end{align*}

It remains to compute $\deg(\eta)$. Using Lemma \ref{lem cusps}, we find for $N \neq 1,2,4$:
\begin{align*}
\deg(\eta) & = \frac{N}{12} \cdot \frac12 \sum_{k \in \Z/N\Z} \frac{\varphi(\gcd(k,N))}{\gcd(k,N)} \\
& = \frac{N}{24} \sum_{d \mid N} \sum_{\substack{k \in \Z/N\Z \\ \gcd(k,N) = d}} \frac{\varphi(d)}{d} \\
& = \frac{N}{24} \sum_{d \mid N} \frac{\varphi(d)}{d} \varphi(N/d).
\end{align*}
The arithmetic function $F(N) = \sum_{d \mid N} \frac{\varphi(d)}{d} \varphi(N/d)$ is the Dirichlet convolution of two multiplicative functions, hence is multiplicative. Moreover, for a prime $p$ and an integer $a \geqslant 1$, we have
\begin{align*}
F(p^a) & = \varphi(p^a) + \frac{\varphi(p^a)}{p^a} + \sum_{m=1}^{a-1} \frac{\varphi(p^m)}{p^m} \varphi(p^{a-m}) \\
& = p^{a-1}(p-1) + \frac{p-1}{p} + \frac{p-1}{p} (p^{a-1}-1) \\
& = p^a \Bigl(1 - \frac{1}{p^2}\Bigr).
\end{align*}
It follows that
\begin{equation*}
\deg(\eta) = \frac{N^2}{24} \prod_{p \mid N} \Bigl(1-\frac{1}{p^2}\Bigr).
\end{equation*}
This gives the desired formula for the index.
\end{proof}

\begin{remark}
It would be interesting to find a conceptual explanation for the fact that the matrix $M$ in the above proof, giving the divisors of the Siegel units, is essentially a symmetric matrix (it becomes symmetric after removing the first row and scaling the other ones).
\end{remark}

It remains to compute the index of $\dv(U)$ in $\dv(S)$. This is the crucial step where we use Streng's result, which provides a description of the group of modular units $\mathcal{O}(Y_1(N))^\times$ in terms of Siegel functions \cite[Theorem 1.2]{Str23}.

\begin{proposition} \label{pro index 2}
For $N \geqslant 4$, we have $(\dv(S) : \dv(U)) = 12 \gcd(2,N) N$.
\end{proposition}

\begin{proof}
Note that $(\dv(S) : \dv(U)) = (S:U)$ since $\dv \colon U_\Q \to \Div^0(C)_\Q$ is an isomorphism. Let $r = \lfloor \frac{N}{2} \rfloor$. We will use the Siegel functions $H_1, \ldots, H_r$ from \cite[(1.4)]{Str23}. Streng \cite[Theorem 1.2]{Str23} proves that $H_1,\ldots,H_r$ generate the group $\mathcal{O}(Y_1(N))^\times/\Q^\times$. He actually uses a different description of the complex points $Y_1(N)(\C)$, namely the map $\nu' : \Gamma^1(N) \backslash \h \to Y_1(N)(\C)$ induced by $\tau \mapsto (\C/(\Z\tau+\Z), \tau/N)$, where $\Gamma^1(N)$ is the transpose of $\Gamma_1(N)$. This is related to our map $\nu$ by means of the formula $\nu'(\tau) = \nu(-1/\tau)$, see \cite[Section 2.3]{Str23}. This will not affect the computation of the index below, but for clarity we introduce the shorthand notation $\tilde{H}_a(\tau) = H_a(-1/\tau)$.

Applying \cite[Lemma 2.13, (6) and (7)]{Str23} with $M = \begin{pmatrix} 0 & 1 \\ -1 & 0 \end{pmatrix}$, we obtain $\tilde{H}_a = \varepsilon_a g_{0,a}$ for some root of unity $\varepsilon_a$. In particular $\tilde{H}_a$ and $g_{0,a}$ have the same image in $U_\Q$. By \cite[Theorem 1.2]{Str23}, every modular unit $u$ in $\mathcal{O}(Y_1(N))^\times$ can be written uniquely as $u = c \tilde{H}_1^{e_1} \cdots \tilde{H}_r^{e_r}$ for some constant $c \in \Q^\times$ and exponents $e_1, \ldots, e_r \in \Z$ satisfying the congruences
\begin{equation*}
\sum_{a=1}^r e_a \equiv 0 \pmod{12} \qquad \textrm{and} \qquad \sum_{a=1}^r a^2 e_a \equiv 0 \pmod{N'},
\end{equation*}
where $N' = \gcd(2,N) N$. Therefore $U$ is contained in $S$, and the index $(S:U)$ is the cardinality of the image of the $\Z$-linear map
\begin{equation*}
\psi \colon \Z^r \to \Z/12\Z \oplus \Z/N'\Z, \qquad (e_1, \ldots, e_r) \mapsto \Bigl( \sum_{a=1}^r e_a, \sum_{a=1}^r a^2 e_a \Bigr).
\end{equation*}
Assume $N \geqslant 6$. Then $r \geqslant 3$ and $\psi$ maps the first three vectors of the canonical basis of $\Z^r$ to $(1,1)$, $(1,4)$ and $(1,9)$. Since these vectors already span $\Z^2$, the map $\psi$ is surjective. In the cases $N=4,5$, the surjectivity of $\psi$ can be checked directly.
\end{proof}

We finally prove the main theorem.

\begin{proof}[Proof of Theorem \ref{main thm C1N}]
Putting together \eqref{eq DSU}, Propositions \ref{pro index 1} and \ref{pro index 2}, we have
\begin{equation} \label{eq C1QN D2N}
|\CC_1^\Q(N)| = |D_{2,N}| \cdot \frac{\gcd(2,N)^2 2^5 3^2}{2^{\lfloor N/2 \rfloor+1} N} \prod_{\substack{d \mid N \\ d \geqslant 3}} d^{\varphi(d)/2} \cdot \prod_{p \mid N} \frac{p^2}{p^2-1},
\end{equation}
where $D_{2,N}$ is given by Theorem \ref{thm formula DkN}. Collecting the powers of $2$ in the expressions yields the factor $2^{N-5}$ in the denominator. This gives the formula of Theorem \ref{main thm C1N} up to sign. To remove the sign ambiguity, we note that $|\CC_1^\Q(N)|$ is positive, so it suffices to show that the product over $\chi \in P_N^+$ in the formula of Theorem \ref{main thm C1N} is a positive real number. Note that the term corresponding to $\overline{\chi}$ is the complex conjugate of that of $\chi$, so that for a non-real $\chi$, the two terms combine to give a positive real number. It remains to treat the case where $\chi$ is real. Note that $1-\overline{\chi}(p) p^{-2}$ is always positive. We will show $B_{2,\chi}>0$. For $\chi = 1$, we have $B_{2,\chi} = B_2 = 1/6$. For $\chi \neq 1$, we can apply \cite[Theorem 9.6]{AIK14}\footnote{Note that the generalised Bernoulli number is normalised there as $N_\chi B_{2,\chi}$.} to obtain
\begin{equation*}
L(\chi,2) = \frac{\pi^2}{N_\chi} g(\chi) B_{2,\chi}.
\end{equation*}
The Gauss sum $g(\chi)$ is positive by \cite[Proposition 8.4]{AIK14}, noting that the quadratic field attached to $\chi$ is real since $\chi$ is even. Finally, the expression of $L(\chi,2)$ as an Euler product shows that $L(\chi,2)>0$, and therefore $B_{2,\chi}>0$.
\end{proof}

\section{Higher weight speculations} \label{sec higher}

As we have seen in the previous section, the determinant $D_{2,N}$ is related to the order of the rational cuspidal class group of $X_1(N)$ (see the formula \eqref{eq C1QN D2N}). We will now introduce a tentative analogue of the cuspidal class group in arbitrary weight $k \geqslant 2$, whose order we expect to be related to $D_{k,N}$.

The analogues of modular units in higher weight are given by Beilinson's Eisenstein symbols \cite{Bei86}. Let $k \geqslant 2$ and $N \geqslant 5$ be integers, and let $E_1(N)$ be the universal elliptic curve over $Y_1(N)$. The Eisenstein symbols of weight $k$ live in the motivic cohomology group $H^{k-1}_{\mathcal{M}}(E_1(N)^{k-2}, \Q(k-1))$, where $E_1(N)^{k-2}$ denotes the fibre product over $Y_1(N)$. The latter group is isomorphic to the higher Chow group $\CH^{k-1}(E_1(N)^{k-2}, k-1)_\Q$. Moreover, Deligne \cite[Lemmas 5.4 and 5.5]{Del71} and Scholl \cite[Section 3]{Sch90} have constructed a canonical smooth compactification $W_{k,N}$ of $E_1(N)^{k-2}$, with a morphism $\pi : W_{k,N} \to X_1(N)$.\footnote{This was originally done for the modular curve $Y(N)$ with $N \geqslant 3$, but works as soon as the moduli problem for generalised elliptic curves with given level structure is representable, see \cite[Appendix]{BDP13}. This applies to the modular curve $Y_1(N)$ with $N \geqslant 5$.} Let $W_{k,N}^\infty = \pi^{-1}(C)$, where $C$ is the set of cusps of $X_1(N)$.

The analogue of the order of vanishing is the residue map, which can be described as follows. By the localisation sequence for higher Chow groups \cite[(iii), p.~269]{Blo86}, \cite{Blo94}, there is an exact sequence
\begin{equation} \label{loc exact sequence}
\begin{tikzcd}
\CH^{k-1}(E_1(N)^{k-2},k-1) \arrow[r, "\partial"] & \CH^{k-2}(W_{k,N}^\infty, k-2) \arrow[r, "\alpha"] & \CH^{k-1}(W_{k,N}, k-2),
\end{tikzcd}
\end{equation}
where $\partial$ is called the residue map. For $k=2$, the sequence reduces to
\begin{equation*}
\begin{tikzcd}
\CH^1(Y_1(N),1) \arrow[r, "\partial"] \arrow[d, "\cong"] & \CH^0(C,0) \arrow[r, "\alpha"] \arrow[d, equal] & \CH^1(X_1(N),0) \arrow[d, "\cong"] \\
\mathcal{O}(Y_1(N))^\times \arrow[r, "\mathrm{div}"] & \Div(C) \arrow{r} & \mathrm{Pic}(X_1(N)).
\end{tikzcd}
\end{equation*}
After replacing $\Div$ by $\Div^0(C)$ and $\mathrm{Pic}(X_1(N))$ by $\mathrm{Pic}^0(X_1(N)) \cong J_1(N)(\Q)$, this is exactly the sequence \eqref{eq seq cuspidal}.

The Eisenstein symbols actually live in a certain piece of the cohomology of $E_1(N)^{k-2}$, and we need to apply a certain projector to the sequence \eqref{loc exact sequence} in order to get a correct analogue of the cuspidal class group. Let us recall the definition of this projector. Let $P : Y_1(N) \to E_1(N)$ be the canonical section of order $N$ of the universal elliptic curve. The translation by $P$ defines an action of $\Z/N\Z$ on $E_1(N)$. The group $\{\pm 1\}$ also acts on $E_1(N)$ by means of the elliptic involution $Q \mapsto -Q$. Finally, the symmetric group $\mathcal{S}_{k-2}$ acts on $E_1(N)^{k-2}$ by permuting the factors. Combining these actions, we obtain an action of the semi-direct product $G = (\Z/N\Z \rtimes \{\pm 1\})^{k-2} \rtimes \mathcal{S}_{k-2}$ on $E_1(N)^{k-2}$. This action extends to $W_{k,N}$ and leaves $W_{k,N}^\infty$ stable. This gives an action on the corresponding higher Chow groups and \eqref{loc exact sequence} becomes an exact sequence of $G$-modules. Let $\varepsilon : G \to \{\pm 1\}$ be the character which is trivial on $(\Z/N\Z)^{k-2}$, is the product on $\{\pm 1\}^{k-2}$, and is the sign character on $\mathcal{S}_{k-2}$. Given a $G$-module $M$, denote by $M_\varepsilon = M / \langle g \cdot m - \varepsilon(g) m : m \in M, g \in G \rangle$ the $\varepsilon$-isotypic quotient of $M$. For a morphism of $G$-modules $f : M \to N$, we denote by $f_\varepsilon : M_\varepsilon \to N_\varepsilon$ the map induced by $f$.

\begin{definition}
Let $k \geqslant 3$ be an integer. The weight $k$ rational cuspidal class group $\CC_1^{k, \Q}(N)$ of $X_1(N)$ is the image of the map
\begin{equation*}
\alpha_\varepsilon : \CH^{k-2}(W_{k,N}^\infty, k-2)_\varepsilon \longrightarrow{} \CH^{k-1}(W_{k,N}, k-2)_\varepsilon
\end{equation*}
\end{definition}

This definition is inspired by Beilinson's construction of the Eisenstein symbol \cite[Section 3]{Bei86}, which uses a closely related residue map. Let $\mathcal{E}_1(N)$ be the identity component of the Néron model of $E_1(N)$ over $X_1(N)$, and let $\mathcal{E}_1(N)^\infty = \mathcal{E}_1(N) \backslash E_1(N)$. The symmetric group $\mathcal{S}_{k-1}$ acts on $E_1(N)^{k-2}$ by identifying $E_1(N)^{k-2}$ with the kernel of the sum map $E_1(N)^{k-1} \to E_1(N)$. Consider the residue map in motivic cohomology
\begin{equation*}
\mathrm{Res} : H^{k-1}_{\mathcal{M}}(E_1(N)^{k-2}, \Q(k-1))_{\mathrm{sgn}} \to H^{k-2}_{\mathcal{M}}((\mathcal{E}_1(N)^\infty)^{k-2}, \Q(k-2))_{\mathrm{sgn}},
\end{equation*}
where $(\cdot)_{\mathrm{sgn}}$ denotes the sign eigenspace for the action of $\mathcal{S}_{k-1}$. Beilinson constructs an explicit right inverse of $\mathrm{Res}$, showing in particular that $\mathrm{Res}$ is surjective. This leads us to expect that the cokernel of $\partial_\varepsilon$ is finite, and thus that $\CC_1^{k, \Q}(N)$ is finite. Moreover, the residues of the Eisenstein symbols involve the Bernoulli polynomial $B_k(x)$ evaluated at $\frac{1}{N}\Z$, see \cite{SS91}, \cite[Section 7]{DK11}. One may therefore anticipate that, as in the weight $2$ case, the order of $\CC_1^{k, \Q}(N)$ is related to $D_{k,N}$.

One could, of course, follow \cite{Bei86, SS91} and use the identity component of the Néron model to define the higher weight cuspidal class group. On the other hand, the torsion in étale motivic cohomology of smooth projective varieties over algebraically closed fields can be compared to étale cohomology \cite[Proposition 3.1(a)]{RS16}. In the present situation, the relevant étale cohomology group is $H^{k-1}_{\textrm{ét}}(W_{k,N,\overline{\Q}}, \Q_\ell/\Z_\ell(k-1))$ and is related to the Galois representations associated with cusp forms of weight $k$ on $\Gamma_1(N)$ \cite{Sch90}. This makes the choice of $W_{k,N}$ a natural one.

Using a method of Bloch, one can construct special elements $S_{P,1}, \ldots, S_{P,N-1}$ in $\CH^2(E,2)$ \footnote{For an elliptic curve $E$ over a field $F$, the group $\CH^2(E,2)$ is isomorphic to the so-called tame $K_2$ group $K_2^T(E)$, defined as the kernel of the tame symbol map $K_2(F(E)) \to \oplus_{P \in E(\overline{F})} F(P)^\times$, see the proof of \cite[Proposition 11.4]{DGdJK26} for an argument.} for elliptic curves $E$ having a rational point $P$ of order $N$ (see \cite[Section 4.1]{BdJLRV24}). One may hope that these elements extend to global elements in $\CH^2(E_1(N),2)$. If this is true, it is likely that after tensoring with $\Q$, these elements are related to the Eisenstein symbols and that their images under the residue map $\partial$ in \eqref{loc exact sequence} can be expressed in terms of the Bernoulli polynomial $B_3(x)$. Some evidence for this is given by the main result of \cite{BdJLRV24} (Theorem 4.13), where the determinant $D_{3,N}$ enters the formula for the limit behaviour at the cusps of the Beilinson regulator of the elements $S_{P,1}, \ldots, S_{P, \lfloor (N-1)/2 \rfloor}$ for $N=7, 8, 10$. It would be interesting to interpret this phenomenon in terms of the residue map and deduce information about the higher weight cuspidal class group.


\end{document}